\documentclass{amsart}
\usepackage{amssymb, latexsym}

\newcommand{\beqa}{\begin{eqnarray*}}
\newcommand{\eeqa}{\end{eqnarray*}}
\newcommand{\beqn}{\begin{eqnarray}}
\newcommand{\eeqn}{\end{eqnarray}}

\newcommand{\bQ}{\mathbb Q}
\newcommand{\B}{\mathbb B}
\newcommand{\C}{\mathbb C}

\newcommand{\R}{\mathbb R}

\newcommand{\N}{\mathbb N}

\newcommand{\mcB}{\mathcal B}

\newcommand{\mcD}{\mathcal D}
\newcommand{\mcE}{\mathcal E}

\newcommand{\mcS}{\mathcal S}

\newcounter{cnt1}
\newcounter{cnt2}
\newcounter{cnt3}
\newcommand{\blr}{\begin{list}{$($\roman{cnt1}$)$}
 {\usecounter{cnt1} \setlength{\topsep}{0pt}
 \setlength{\itemsep}{0pt}}}
\newcommand{\bla}{\begin{list}{$($\alph{cnt2}$)$}
 {\usecounter{cnt2} \setlength{\topsep}{0pt}
 \setlength{\itemsep}{0pt}}}
\newcommand{\bln}{\begin{list}{$($\arabic{cnt3}$)$}
 {\usecounter{cnt3} \setlength{\topsep}{0pt}
 \setlength{\itemsep}{0pt}}}
\newcommand{\el}{\end{list}}
\newtheorem{thm}{Theorem}[section]
\newtheorem{lem}[thm]{Lemma}
\newtheorem{cor}[thm]{Corollary}

\newtheorem{Def}[thm]{Definition}

\newtheorem{rem}[thm]{Remark}
\newcommand{\Rem}{\begin{rem} \rm}
\newcommand{\bdfn}{\begin{Def} \rm}
\newcommand{\edfn}{\end{Def}}

\newcommand{\ba}{\begin{array}}
\newcommand{\ea}{\end{array}}

\begin{document}
\begin{center}\large{{\bf{ Approach to the construction of the spaces  $ S{D^p}[\R^{\infty}]$  for $1 \leq p \leq \infty$}}} \\

 
  Hemanta Kalita$^{1}$ and   Bipan Hazarika$^{2,\ast}$\\

$^{1}$Department of Mathematics, Patkai Christian College (Autonomous), Dimapur,\\ Patkai 797103, Nagaland, India.\\

$^{2}$Department of Mathematics, Gauhati University, Guwahati 781014, Assam, India\\
Email:  hemanta30kalita@gmail.com; bh\_rgu@yahoo.co.in;  bh\_gu@gauhati.ac.in
\end{center}
\title{}
\author{}
\thanks{$^{\ast}$The corresponding author}
\thanks{\today} 

\begin{abstract} The objective of this paper is to construct an extension of the  class of Jones distribution Banach spaces $SD^p[\R^n], 1\le p\le \infty,$ which appeared in the book by Gill and Zachary \cite{TG} to  $S{D^p}[\R^{\infty}]$  for $1\leq p \leq \infty.$ These spaces are separable Banach spaces, which contain  the Schwartz distributions as continuous dense embedding. These spaces provide a Banach space structure for Henstock-Kurzweil integrable functions that is similar to the Lebesgue spaces for Lebesgue integrable functions. 
{\footnotesize {\bf{AMS subject classification \textrm{(2010)}:}}} 46B03, 46B20, 46B25.
\end{abstract}
\maketitle

\maketitle

\pagestyle{myheadings}
\markboth{\rightline {\scriptsize   Kalita and Hazarika}}
        {\leftline{\scriptsize  Approach to the construction of the spaces \dots}}

\maketitle
\section{Introduction and Preliminaries}
The  theory of distributions is based on the action of linear functionals on a space of
test functions.  In the original approach of Schwartz \cite{SC} both test functions and linear functional have a natural topological vector space structure which are not normable. Sobolev gave  powerful reply for this using functions that were Lebesgue type. At the same time Lebesgue integrable functions have some limitations. In Physics particularly in quantum theory Banach space structure needed with non-absolute integrable functions. 
 Gill and Zachary  gave more strong reply than Sobolev (see \cite{TG,GL}) by introducing the family of strong Jones function spaces $S{D^p}[\R^n]$ for $ 1 \leq p \leq \infty,$  which contain the non-absolute integrable functions.   Henstock-Kurzweil integral was first developed by R. Henstock and J. Kurzweil  from Riemann integral with the concept of tagged partitions and gauge functions. Henstock-Kurzweil integral (HK-integral) is a kind of non-absolute integral and contain Lebesgue integral (we refer \cite{AB,AB1,AB2,VB,RH,RH1,ET,TY}).       The most important of the finitely additive measure is the one that generated by HK-integral, which generalize the Lebesgue, Bochner and Pettis integrals, for instance see \cite{AL,RAG,RH,CS}.        As a major drawback of HK-integrable functions space is that it is not naturally a Banach space. 
 
       Y. Yamasaki \cite{YY} developed a theory of Lebesgue measure on $\R^\infty.$ In \cite{GM}, Gill and Myres introduced a theory of Lebesgue measure on $\R^\infty :$ the construction of which almost same as the Lebesgue measure on $\R^n.$
       Throughout our paper, we suppose the notation $\mathbb{R}_{I}^{\infty}$ and assume that $I$ is understood. 
        In this paper, we will focus  on the main class of Banach spaces $SD^p[\mathbb{R}^{\infty}],~1\leq p \leq \infty.$
   \begin{Def}\cite{CS}
   A function $ f:[a,b] \to \mathbb{R} $ is HK-integrable if there exists a function $ F:[a,b] \to \mathbb{R} $ and for every $ \epsilon > 0 $ there is a function $ \delta(t) > 0 $ such that for any $ \delta$-fine partition $ D=\{[u,v], t \} $ of $I_0=[a,b], $ we have $$|| \sum[f(t)(v-u)-F(u,v)]|| < \epsilon,$$ where the sum $ \sum $ is  run over $ D= \{ ([u,v], t) \}$ and $F(u,v)= F(v)-F(u).$     We write $ HK\int_{I_0} f=F(I_0).$
   \end{Def}
 \begin{Def}\cite[Definition 2.5]{TG} 
   Let $ A_n = A \times I_n$ and $ B_n = B \times I_n $ ($n^{th}$ order box sets in $ \mathbb{R}^\infty$). We define 
   \begin{enumerate}
   \item[(a)] $ A_n \cup B_n =(A\cup B) \times I_n;$
   \item[(b)] $ A_n \cap B_n = (A \cap B ) \times I_n; $
   \item[(c)] $ B_{n}^c = B^c \times I_n.$
   \end{enumerate}
   \end{Def}
 \begin{Def}
  \cite[Definition 1.11]{GM} We define $\mathbb{ R}_{I}^{n} =\mathbb{ R}^n \times I_n.$ If $ T $ is a linear transformation on $ \mathbb{R}^n $ and $ A_n = A \times I_n ,$ then $ T_I $ on $ \mathbb{R}_{I}^{n} $ is defined by $ T_I[A_n] = T[A] $. We also define $ B[\mathbb{R}_{I}^{n}]$ to be the Borel $\sigma$-algebra for $ \mathbb{R}_{I}^{n},$ where the topology on $ \mathbb{R}_{I}^{n} $ is defined via the class of open sets $ D_n = \{ U \times I_n : U $ is open in $ \mathbb{R}^n\}.$ For any $ A \in B[\mathbb{R}^n], $ we define $\lambda_\infty(A_n) $ on $ \mathbb{R}_{I}^{n} $ by product measure $ \lambda_\infty(A_n) = \lambda_n(A) \times \Pi_{i= n +1 }^{\infty} \lambda_I(I) = \lambda_n(A).$
  \end{Def}
   \begin{thm} \cite[Theoem 2.7]{TG}
 $ \lambda_{\infty}(.) $ is a measure on $ B[\mathbb{R}_{I}^{n}] $ is equivalent to $n$-dimensional Lebesgue measure on $ \mathbb{R}^n.$
   
 \end{thm}
\begin{cor}\cite[Corollary 2.8]{TG}
  The measure $ \lambda_{\infty}(.) $ is both translationally and rotationally invariant on $ (\mathbb{R}_{I}^{n}, B[\mathbb{R}_{I}^{n}]) $ for each $ n \in \mathbb{N}.$
\end{cor}
   Recalling the  theory on $ \mathbb{R}_{I}^{n} $ that completely paralleis that on $ \mathbb{R}^n. $     Since $ \mathbb{R}_{I}^{n} \subset \mathbb{R}_{I}^{n+1},$ we have an increasing sequence, so we define $ \widehat{\mathbb{R}}_{I}^{\infty} = \lim\limits_{ n \to \infty} \mathbb{R}_{I}^{n} = \bigcup\limits_{k=1}^{\infty} \mathbb{R}_{I}^{k}.$ Let $X_1= \widehat{\R}_I^\infty $ and let $\tau_1$ be the topology induced by the class of open sets $D \subset X_1$ such that $D= \bigcup\limits_{n=1}^{\infty}D_n = \bigcup\limits_{n=1}^{\infty}\{U \times I_n : U \mbox{~is~open~ in~}\R^n\}.$ Let $X_2= \R^\infty \setminus \widehat{\R}_I^\infty,$ and let $\tau_2$ be discrete topology on $X_2$ induced by the discrete metric so that, for $x,y \in X_2,~x \neq y,~d_2(x,y)=1$ and for $x=y,~d_2(x,y)=0.$
   \begin{Def}
   \cite[Definition 2.9]{TG} We define $(\R_I^\infty, \tau) $ be the co-product $(X_1, \tau_1) \otimes (X_2, \tau_2)$ of $(X_1, \tau_1) $ and $(X_2, \tau_2),$ so that every open set in $(\R_I^\infty, \tau)$ is the disjoint union of two open sets $G_1 \cup G_2$ with $G_1 $ in $(X_1, \tau_1)$ and $G_2 $ in $(X_2, \tau_2)$. 
   \end{Def}
   It follows that $\R_I^\infty = \R^\infty$ as sets. However, since every point in $X_2$ is open and closed in $\R_I^\infty$ and no point is open and closed in $\R^\infty,$ So, $\R_I^\infty \neq \R^\infty$ as a topological spaces.    In \cite {GM}, Gill and Myres  shown that it  can extend the measure $ \lambda_{\infty}(.) $ to $ \mathbb{R}^\infty.$    
 
Similarly, if $B[\R _I^n ]$ is the Borel $\sigma$-algebra for $\R _I^n ,$ then $B[\R_I^n ] \subset B[\R _I^{n+1}] $ by $$\widehat{B}[\R_I^\infty]  = \lim\limits_{n \to \infty} B[\R_I^n ] =\bigcup\limits_{k=1}^{\infty}B[\R_I^k].$$ Let $B[\R_I^\infty] $ be the smallest $\sigma$-algebra containing $\widehat{B}[\R_I^\infty] \cup P(\R^\infty\setminus \bigcup\limits_{k=1}^{\infty}[\R_I^k]),$ where $P(.)$ is the power set.
It is obvious that the class $B[\R _I^\infty ]$ coincides with the Borel $\sigma$-algebra generated by the $\tau$-topology on $\R_I ^\infty.$
\begin{lem}\cite[Lemma 1.15]{GM}
$\widehat{B}[\R_I^\infty] \subset B[\R_I^\infty]$
\end{lem}
 \subsection{Measurable function}
  We discuss about measurable function on $\R_I^\infty.$ 
     Let $ x =(x_1,x_2,\dots) \in \mathbb{R}_{I}^{\infty},$  $ I_n = \Pi_{k=n+1}^{\infty}[\frac{-1}{2}, \frac{1}{2}] $ and let $ h_n(\widehat{x})= \chi_{I_n}(\widehat{x}),$ where $\widehat{x} = (x_i)_{i=n+1}^{\infty}.$
     \begin{Def}
     \cite[Definition 2.46]{TG} Let $M^n $ represented the class of measurable functions on $\R^n.$ If $ x \in \R_I^\infty$ and $f^n \in M^n.$ Let $\overline{x}= (x_i)_{i=1}^{n} $ and define an essentially tame measurable function of order $n$ (or $ e_n$-tame) on $\R_I^\infty$ by $$f(x)= f^n(\overline{x}) \otimes h_n(\widehat{x}).$$ We let $M_I^n = \{f(x)~: f(x) = f^n( \overline{x}) \otimes h_n(\widehat{x}),~x \in \R_I^\infty \}$ be the class of all $e_n$-tame functions.
     \end{Def}
     \begin{Def}\cite[Definition 2.47]{TG}\label{def19}
     A function $f: \R_I^\infty \to \R$ is said to be measurable and we write $ f \in M_I$, if there is a sequence $\{f_n \in M_I^n \} $ of $e_n$-tame functions, such that $$\lim\limits_{ n \to \infty} f_n(x) \to f(x)~\lambda_\infty-(a.e.).$$
     \end{Def}
     The existence of functions satisfying above definition is not obvious. So, we have 
     \begin{thm} \cite[Theorem 2.48]{TG}
     (Existence) Suppose that $ f: \R_I^\infty \to (- \infty, \infty)$ and $f^{-1}(A) \in B[\R_I^\infty]$ for all $ A \in B[\R]$ then there exists a family of functions $\{f_n \},~f_n \in M_I^n $ such that $f_n(x) \to f(x) , \lambda_\infty(-a.e.)$
     \end{thm}
     \begin{rem}
     Recalling that any set $A,$ of non zero measure is concentrated in $X_1$ that is $\lambda_\infty(A)= \lambda_\infty(A \cap X_1) $ also follows that the essential support of the limit function $f(x)$ in Definition \ref{def19}, i.e. $\{x:~f(x) \neq 0 \}$ is concentrated in $\R_I^N,$ for some $N.$
     \end{rem}
     \subsection{Integration theory on $\R_I^\infty$}
      We discuss a constructive theory of integration on $\R_I^\infty$ using the known properties of integration on $\R_I^n.$ This approach has the advantages that all the theorems for Lebesgue measure apply. Proofs are similar as for the proof on $\R^n.$ Let $L^1[\R_I^n]$ be the class of integrable functions on $\R_I^n.$ Since $L^1[\R_I^n] \subset L^1[\R_I^{n+1}]$, we define $L^1[\widehat{\R}_I^\infty]= \bigcup\limits_{n=1}^{\infty}L^1[\R_I^n].$
      \begin{Def}\cite[Definition 3.13]{GM}
      We say that a measurable function $ f \in L^1[\R_I^\infty],$ if there is a Cauchy-sequence $\{f_n \} \subset L^1[\widehat{\R}_I^\infty]$ with $f_n \in L^1[\R_I^n] $ and $$\lim\limits_{n \to \infty}f_n(x) = f(x),~\lambda_\infty-(a.e.)$$
      \end{Def}
      \begin{thm}
      $L^1[\R_I^\infty] = L^1[\widehat{\R}_I^\infty].$
      \end{thm}
      \begin{proof}
      We know that $ L^1[\R_I^n] \subset L^1[\widehat{\R}_I^\infty]$ for all $n.$ We sufficiently need to prove $L^1[\widehat{\R}_I^\infty]$ is closed. Let $ f $ be a limit point of $L^1[\widehat{\R}_I^\infty]~(f \in L^1[\R_I^\infty]).$ If $f=0$ then the result is obvious. So we consider $f \neq 0.$ If $A_f$ is the support of $f,$ then $\lambda_\infty(A_f)= \lambda_\infty(A_f \cap X_1).$ Thus $A_f \cup X_1 \subset \R_I^N $ for some $N.$ This means that there is a function $g \in L^1[\R_I^{N+1}] $ with $ \lambda_\infty(\{x:~f(x) \neq g(x) \})=0.$ So, $f(x)=g(x)$-a.e. As $L^1[\R_I^n]$ is a set of equivalence classes. So,  $L^1[\R_I^\infty] = L^1[\widehat{\R}_I^\infty].$
      \end{proof}
      \begin{Def}\cite[Definition 3.14]{GM}
      If $f \in L^1[\R_I^\infty],$ we define the integral of $f$ by $$ \int_{\R_I^\infty} f(x) d \lambda_\infty(x)= \lim\limits_{n \to \infty} \int_{\R_I^\infty} f_n(x) d \lambda_\infty(x),$$ where $\{f_n \} \subset L^1[\R_I^\infty] $ is any Cauchy-sequence converging to $f(x)$-a.e.
      \end{Def}
      \begin{thm}\cite[Theorem 3.15]{GM}
      If $f \in L^1[\R_I^\infty] $ then the above integral exists and all theorems that are true for $f \in L^1[\R_I^n],$ also hold for $ f \in L^1[\R_I^\infty].$
      \end{thm}     
     We denote $\mathbb{N}_{0}^{\infty}$ be the set of all multi-index infinite tuples $\alpha=(\alpha_1,\alpha_2,\dots),$ with $\alpha_i\in \mathbb{N}$ and all but a finite number of entries are zero (also see \cite{TG}). 
     \begin{Def}  The Schwartz space $\mathcal{S}[\mathbb{R}_I^\infty] $ is the topological space of functions $f: \mathbb{R}_I^\infty \to \mathbb{C}$ such that $ f \in C^\infty[\mathbb{R}_I^\infty] $ and $x^{\alpha}\partial^{\beta}f(x) \to 0 $ as $|x| \to \infty $ for every pair of multi-indices $ \alpha, \beta \in \mathbb{N}_{0}^{\infty}, $ and $ f \in \mathcal{S}[\mathbb{R}_I^\infty],$ let $$||f||_{\alpha, \beta}= \sup\limits_{\mathbb{R}_I^\infty}|x^{\alpha}\partial^{\beta}f|.$$
     \end{Def}      
      If $ x=(x_1, x_2,\dots) \in \mathbb{R}_I^\infty$ and $\alpha \in \mathbb{N}_{0}^{\infty},~\alpha=(\alpha_1, \alpha_2,\dots),$ we define  $x^{\alpha}=\Pi_{k=1}^{\infty}x_k^{\alpha_k}$ a product of real or complex numbers.\\
     A sequence of functions $\{f_k~: k \in \mathbb{N}\} $ converges to a function $ f $ in $\mathcal{S}[\mathbb{R}_I^\infty] $ if $$||f_n -f||_{\alpha, \beta} \to 0 $$ as $k \to \infty$ for every $ \alpha, \beta \in \mathbb{N}_{0}^{\infty}$.\\
     That is the Schwartz space consists of smooth functions whose derivatives decay at infinity faster than any power. For details on Schwartz space and distribution functions we refer \cite{MA,TG,LG,KB,SY}.
     \begin{thm}
     $\mathcal{S}[\R_I^\infty]$ (respectively $\mathcal{S}^{'}[\R_I^\infty]$) is a Fr\'echet space, which is dense in $\C_0[\R_I^\infty].$
     \end{thm}
     \begin{proof}
     The proof is similar to (p 90 Theorem 2.88 of \cite{TG}).
          \end{proof}
     \begin{Def}
     A tempered distribution $ T $ on $\mathbb{R}_I^\infty$ is a continuous linear functional $ T : \mathcal{S}[\mathbb{R}_I^\infty] \rightarrow \mathbb{C}.$ The topological vector space of tempered distributions is denoted by $\mathcal{S}^{'}[\mathbb{R}_I^\infty]$ or $ \mathcal{S}^{'}.$ If $< T, f > $ denotes the value of $ T \in \mathcal{S}^{'}$ acting on $ f \in \mathcal{S},$ then a sequence $\{T_k\} $ converges to $ T$ in $\mathcal{S}^{'}.$ Written $ T_k \to T $ if $\lim\limits_{k\to \infty} < T_k, f > = < T, f > $ for every $ f \in \mathcal{S}.$
     \end{Def}
     \section*{Purpose of the paper:}
     The purpose
of this paper is to introduce a  class of Banach spaces on $\R_I^\infty$ which contain
the non-absolutely integrable functions, but also contain the Schwartz test
function spaces as dense and continuous embedding.
 \section{The Jones family of spaces $SD^p[\R_I^\infty]~1\leq p \leq \infty$} 
   The theory of distributions is based on the action of linear functional on a space of test function. In \cite{TG}, Gill and Zachary introduced another class of Banach spaces which contain the non-absolutely integrable functions, but also contains the Schwartz test function space as a dense and continuous embedding.
   \begin{lem}
   (Kuelbs lemma) If $\mcB$ is a separable Banach space, there exists a separable Hilbert space $\mcB \subset H $ as continuous dense embedding 
   \end{lem}
   \begin{proof}
Let $\{e_k\}$ be a countable dense sequence on the unit ball $B$ and let $\{e_{k}^{*}\}$ be any fixed set of corresponding duality mappings (i.e for each $k, e_{k}^{*} \in \mcB^*$ and $e_k^*(e_k)= < e_k, e_k^*> = ||e_k||_{\mcB}^{2}= ||e_k^*||_{\mcB^*}^{2}=1.)$ For each $k, $ let $t_k= \frac{1}{2^k}$ and define $(u,v) $ as follows: $$(u,v)= \sum_{k=1}^{\infty}t_k e_k^*(u) e_{k}^{-*}(v) = \sum_{k=1}^{\infty} \frac{1}{2^k}e_k^*(u) e_{k}^{-*}(v)$$ It is clear that $(u,v) $ is an inner product on $\mcB.$ Let $H$ be the completion of $\mcB$ with respect to this inner product. It is clear that $\mcB$ is dense in $H,$ and 
\begin{align*}
||u||_H^2 &= \sum_{k=1}^{\infty} t_k|e_k^*(u)|^2 \\&\leq \sup|e_k^*(u)|^2 \\&=||u||_\mcB^2
\end{align*}
So, the embedding is continuous. Now note that if $\mcB= L^1[\R^n]$,$$|e_k^*(u)|^2 = |\int_{\R^n} e_k^*(x)u(x)d \lambda_n(x)|^2$$ where $e_k^*(x) \in L^\infty[\R^n].$ 
\end{proof}
It is clear that the Hilbert space H, will contain some non-absolutely integrable
functions, but we cannot say which ones will or will not be in there. This gave Steadman the needed hint for her Hilbert space.
   To construct the space we remembering the remarkable Jone's functions of $3.3.2 $ of \cite{TG} in $C_c^\infty[\R_I^n].$\\
 Fix $n$ and let $\bQ_I^n$ be the set $\{x \in \R_I^n \}$ such that the first co-ordinates $(x_1, x_2,..,x_n)$ are rational. Since this is a countable dense set in $\R_I^n,$ we can arrange it as $\bQ_I^n=\{x_1,x_2,...\}$ for each $k$ and $i$, let $\mcB_k(x_i)$ be the closed cube centered at $x_i$ with edge $e_k= \frac{1}{2^{k-1} \sqrt n},~x \in \N$. Now choose the natural order which maps $\N \times \N$ bijectively to $\N,$ and let $\{\mcB_k:~k \in \N\}$ be the resulting set of (all) closed cubes $$\{\mcB_k(x_i)|~(k,i) \in \N\times\N \}$$ and each $x_i \in \bQ_I^n.$ For each $x \in \mcB_k,~x=(x_1,x_2,..,x_n)$ we define $\mcE_k(x) $ by $\mcE_k(x)= (\mcE_k^i(x_1), \mcE_k^i(x_2),..,\mcE_k^i(x_n)$ with $|\mcE_k(x)| < 1, ~x \in \Pi_{j=1}^{n}I_k^i $ and $\mcE_k(x)=0,~x $ not belongs in  $\Pi_{j=1}^{n}I_k^i $. Then $\mcE_k(x) $ is in $L^p[\R_I^n]^n= L^p[\R_I^n]$ for $ 1 \leq p \leq \infty.$ Define $F_k(.)$ on $L^p[\R_I^n]$ by $$F_k(f)= \int_{\R_I^n}\mcE_k(x)f(x)d \lambda_\infty(x)$$ Since each $\mcB_k$ is a cube with sides parallel to the co-ordinate axes, $F_k(.)$ is well defined integrable functions, is a bounded linear functional on $L^p[\R_I^n]$ for each $k,$ with $||F_k||_{\infty} \leq 1 $ and if $F_k(f)=0 $ for all $k,~f=0$ so that $\{F_k\}$ is a fundamental on $L^p[\R_I^n]$ for $ 1 \leq p \leq \infty.$ Let $t_k>0 $ such that $t_k = \frac{1}{2^k}$ so that $ \sum_{k=1}^{\infty} t_k=1$ and defined an inner product $(.)$ on $L^1[\R_I^n]$ by $$(f,g)= \sum_{k=1}^{\infty}[ \int_{\R_I^n}\mcE_k(x) f(x)d \lambda_\infty(x)][\int_{\R_I^n}\mcE_k(y)g(y) d \lambda_\infty(y)]^c$$ The completion of $L^1[\R_I^n]$ with the above inner product is a Hilbert space which we denote $S{D^2}[\R_I^n].$ 
\begin{rem}
$S{D^2}[\R_I^n]$ will contain some class of non absolute integrable functions. Interested researcher can think for this. If we observe \cite{TG} , we must sure HK-integrable function space contained in $S{D^2}[\R_I^n]$ . we are not interested to find that portion in this paper.
\end{rem}
\begin{thm}
For each $p,~1 \leq p \leq \infty$ we have 
\begin{enumerate}
\item The space $ L^p[\R_I^n] \subset S{D^2}[\R_I^n]$ as a continuous, dense and compact embedding.
\item The space $\mathfrak{M}[\R_I^n] \subset S{D^2}[\R_I^n],$ the space of finitely additive measures on $\R_I^n,$ as a continuous dense and compact embedding.
\end{enumerate}
\end{thm}
\begin{proof}
For $(1)$ by construction, $S{D^2}[\R_I^n] $ contains $L^1[\R_I^n]$ densely. So, we need to show that $L^q[\R_I^n] \subset S{D^2}[\R_I^n]$ for $q \neq 1.$ If $f \in L^q[\R_I^n]$ and $ q < \infty,$ we have, $$
||f||_{S{D^2}} \leq C||f||_q$$
Hence, $f \in S{D^2}[\R_I^n].$ For $q= \infty, $ first note that $$vol(\mcB_k)^2 \leq [\frac{1}{\sqrt n}]^{2n} \leq 1 $$ So, we have $$
||f||_{S{D^2}}  \leq ||f||_\infty$$
Thus $f \in S{D^2}[\R_I^n]$ and $L^\infty[\R_I^n] \subset S{D^2}[\R_I^n].$\\
To prove compactness, suppose $\{f_j\}$ is any weakly convergent sequence in $L^p[\R_I^n],~ 1 \leq p \leq \infty$ with limit $f.$ Since $\mcE_k \in L^q,~\frac{1}{p}+ \frac{1}{q}=1$, $$\int_{\R_I^n}\mcE_k(x)[f_j(x)- f(x)] d \lambda_\infty(x) \to 0$$ for each $k.$ It follows that $\{f_j\}$ converges strongly to $f $ in $S{D^2}[\R_I^n].$\\
To prove $(2)$ as $\mathfrak{M}[\R_I^n] = L^[\R_I^n]^{**} \subset S{D^2}[\R_I^n]$.
\end{proof}
\begin{Def}
We call $S{D^2}[\R_I^n]$ the Jones strong distribution Hilbert space on $\R_I^n.$
\end{Def}
Let $\beta$ be a multi-index of non negative integers $$\beta= ( \beta_1, \beta_2,..,\beta_k)$$ with $|\beta|= \sum_{j=1}^{k}\beta_j.$ If $D$ denotes the standard partial differential operator. Let $D^\beta= D^\beta_1 D^\beta_2...D^\beta_k.$
\begin{thm}
Let $D[\R_I^n] $ be $C_c^\infty[\R_I^n] $ equipped with the standard locally convex topology( test functions)
\begin{enumerate}
\item If $\Phi_j \to \Phi$ in $D[\R_I^n],$ then $\Phi_j \to \Phi$ in the norm topology of $S{D^2}[\R_I^n],$ so that $D[\R_I^n] \subset S{D^2}[\R_I^n]$ as continuous dense embedding.
\item If $T \in D^{'}[\R_I^n]$ then $T \in S{D^2}[\R_I^n]^{'}$ so that $D^{'}[\R_I^n] \subset S{D^2}[\R_I^n]^{'}$ as a continuous dense embedding.
\item For any $f,g \in S{D^2}[\R_I^n]$ and multi-index $\beta,~(D^\beta f, g)_{SD}= (-i)^\beta (f,g)_{SD}.$
\end{enumerate}
\end{thm}
\begin{proof}
To prove $(1), $ suppose that $\Phi_j \to \Phi $ in $D[\R_I^n].$ By definition there exists a compact set $K \subset \R_I^n$ which is the support of $\Phi_j - \Phi$ and $D^\beta \Phi_j $ converges to $D^\beta \Phi $ uniformly on $K$ for every multi-index $\beta.$ Let $\{\mcE_{K_l}\}$ be the set of all $\mcE_l,$ with support $K_l \subset K.$ If $\beta$ is a multi-index we have
\begin{align*}
\lim\limits_{j \to \infty}||D^\beta \Phi_j - D^\beta \Phi||_{SD} &= \lim\limits_{j \to \infty}\{\sum_{l=1}^{\infty}t_{K_l}|\mcE_{K_l}(x)[D^\beta \Phi_j(x) - D^\beta \Phi(x)] d \lambda_\infty(x)|^2\}^\frac{1}{2} \\&\leq M\lim\limits_{j \to \infty}\sup_{n \in K}|D^\beta \Phi_j(x)- D^\beta \Phi(x)|=0
\end{align*}
Thus , since $\beta$ is arbitrary, we see that , we see that $D[\R_I^n] \subset S{D^2}[\R_I^n]$ as continuous embedding. Since $C_c^\infty[\R_I^n] $ is dense in $L^1[\R_I^n]$, $D[\R_I^n] $ is dense in $S{D^2}[\R_I^n].$\\
To prove $(2)$, we note that as $D[\R_I^n]$ is a dense locally convex subspace of $S{D^2}[\R_I^n],$ by corollary of Hahn-Banach theorem every continuous linear functional,$T $ defined  on $D[\R_I^n],$ can be extended to a continuous linear functional on $S{D^2}[\R_I^n]$. By Riesz representation theorem, every continuous linear functional $T $ defined on $S{D^2}[\R_I^n] $ is of the form $T(f)= (f,g)_{SD} $ for some $g \in S{D^2}[\R_I^n]. $ Thus $T \in S{D^2}[\R_I^n]^{'}$ and by the identification $T \leftrightarrow g $ for each $T $ in $D^{'}[\R_I^n]$ as  continuous dense embedding.\\
For $(3) $ as each $\mcE_k \in C_c^\infty[\R_I^n]$ so that for any $f \in S{D^2}[\R_I^n]$,$$\int_{\R_I^n}\mcE_k(x).D^\beta f(x) d \lambda_\infty(x) = (-1)^|\beta| \int_{\R_I^n} D^\beta \mcE_k(x).f(x) d \lambda_\infty(x)$$ That is $$(-1)^|\beta| \int_{\R_I^n} D^\beta \mcE_k(x).f(x) d \lambda_\infty(x)= (-i)^{|\beta|} \int_{\R_I^n}\mcE_k(x). f(x)d \lambda_\infty(x)$$ Follows, for any $g \in S{D^2}[\R_I^n],~(D^{\beta} f, g)_{SD^2}= (-i)^|\beta|(f,g)_{SD^2}.$ 
\end{proof}
\subsection{The General case,} $S{D^p}[\R_I^n],~ 1 \leq p \leq \infty.$ To construct $S{D^p}[\R_I^n]$ for all $p $ and for $f \in L^p[\R_I^n],$ define:$$||f||_{SD^p[\mathbb{R}_{I}^{n}]} =\left\{\begin{array}{c}\left(\sum_{|\beta| \leq m}\sum\limits_{k=1}^{\infty}t_k\left|\int_{R_{I}^{n}}\mathcal{E}_k(x)D^{\beta}u(x)d\lambda_{\infty}(x)\right|^p\right)^{\frac{1}{p}}, \mbox{~for~} 1\leq p<\infty;\\
 \sum_{|\beta| \leq m} \sup\limits_{k\geq 1}\left|\int_{R^{\infty}_{I}}\mathcal{E}_k(x)D^{\beta}u(x)d\lambda_{\infty}(x)\right|, \mbox{~for~} p=\infty \end{array}\right.$$It is easy to see that $||.||_{SD^p}^{'}$ defines a norm on $L^p[\R_I^n].$ If $S{D^p}[\R_I^n]$ is the completion of $L^p[\R_I^n]$ with respect to this norm, then we have 
 \begin{thm}
 For each $q,~1 \leq q \leq \infty~~~L^q[\R_I^n] \subset S{D^p}[\R_I^n]$ as a dense continous embeddings.
 \end{thm}
 \begin{proof}
 As $S{D^p}[\R_I^n] $ contains $L^p[\R_I^n] $ densely, so we have to only show that $  L^q[\R_I^n] \subset S{D^p}[\R_I^n] $ for $q \neq p.$\\
 First, suppose that $p < \infty. $ If $ f \in L^q[\R_I^n]$ and $q < \infty,$ we have: $$
 ||f||_{SD^p} \leq ||f||_q $$
 Hence $ f \in S{D^p}[\R_I^n]$ for $q = \infty,$ we have
 \begin{align*}
 ||f||_{SD^p} &= [\sum_{k=1}^{\infty}t_k|\int_{\R_I^n}\mcE_k(x) f(x) d \lambda_\infty(x)|^p]^\frac{1}{p} \\&\leq [[\sum_{k=1}^{\infty}t_k[vol(\mcE_k)]^p][ess \sup|f|]^p]^\frac{1}{p}\\&\leq M||f||_\infty
 \end{align*}
 Thus $f \in S{D^p}[\R_I^n]$ and $L^q[\R_I^n] \subset S{D^p}[\R_I^n]$. The case $p= \infty$ is obvious.
 \end{proof}
 \begin{thm}
 For $S{D^p}[\R_I^n],~1 \leq p \leq \infty,$ we have 
 \begin{enumerate}
 \item $S{D^p}[\R_I^n]$ is uniformly convex.
 \item If $\frac{1}{p}+ \frac{1}{q}=1 $ then the dual space of $S{D^p}[\R_I^n] $ is $S{D^q}[\R_I^n].$
 \item If $K$ is a weakly compact subset of $L^p[\R_I^n],$ it is a strongly compact subset of $S{D^p}[\R_I^n].$
 \item The space $S{D^\infty}[\R_I^n] \subset S{D^p}[\R_I^n].$
 \end{enumerate}
 \end{thm}
 \begin{proof}
 For $(1)$ proof  follows from a modification of the proof of the Clarkson inequality for $l^p$ norms.\\
 For $(2)$ Let $l_{k}^{p}(g) = ||g||_{SD^p}^{2-p}|\int_{\R_I^n}\mcE_k(x) g(x) d \lambda_\infty(x)|^{p-2}$ and observe that for $p \neq 2,~ 1 < p < \infty$ the linear functional$$L_g(f)= \sum_{k=1}^{\infty}t_k l_{k}^{p}(g) \int_{\R_I^n}\mcE_k(x) g(x) d \lambda_\infty(x) \int_{\R_I^n} \mcE_k(y) f^{*}(y)d \lambda_\infty(y)$$ is a duality map on $S{D^q} $ for each $g \in S{D^p}$ and that $S{D^p} $ is reflexive from $(1).$\\
 For $(3)$ If $\{f_m\}$ is any weakly convergent sequence in $K$ with limit $f,$ then $$\int_{\R_I^n}\mcE_k(x)[f_m(x) -f(x)] d \lambda_\infty(x) \to 0$$ for each $k.$ It follows that $\{f_m\}$ converges strongly to $f$ in $S{D^p}.$\\
 For $(4)$ Let $ f \in S{D^\infty}$ implies $|\int_{\R_I^n} \mcE_k(x) f(x) d \lambda_\infty(x)|$ is uniformly bounded for all $k.$ It follows that $|\int_{\R_I^n}\mcE_k(x) f(x) d \lambda_\infty(x)|^p$ is uniformly bounded for each $p,~1 \leq p < \infty.$ It is now clear from the definition of $S{D^\infty}$ that :$$[\sum_{k=1}^{\infty}t_k|\int_{\R_I^n}\mcE_k(x) f(x) d \lambda_\infty(x)|^p]^\frac{1}{p} \leq ||f||_{SD^\infty} < \infty .$$
 \end{proof}
 \begin{thm}
 For each $p,~1 \leq p \leq \infty, $ the test function $D \subset S{D^p}[\R_I^n]$ as continuous embedding.
 \end{thm}
 \begin{proof}
 Since $S{D^\infty}[\R_I^\infty]$ is continuously embedded in $S{D^p}[\R_I^n]~1 \leq q < \infty,$ it suffices to prove the result for $S{D^\infty}[\R_I^n]$\\
 Suppose $\Phi_j \to \Phi $ in $D[\R_I^n]$ so that there exists a compact set $K \subset \R_I^n,$ containing the support of $\Phi_j - \Phi$ and $D^\beta \Phi_j $ converges to $D^\beta \Phi$ uniformly on $K$ for every multi-index $\beta.$ Let $L=\{l \in \N: $the support $\mcE_l,$ stp$\{\mcE_l\} \subset K \},$ then 
 \begin{align*}
 \lim\limits_{j \to \infty}||D^\beta \Phi - D^\beta \Phi_j||_{SD} &= \lim\limits_{j \to \infty} \sup_{l \in L}|\int_{\R_I^n}[ D^\beta \Phi(x) - D^\beta \Phi_j(x)] \mcE_l(x) d \lambda_\infty(x)| \\&\leq vol(\B_l) \lim\limits_{j \to \infty} \sup_{ x \in K}|D^\beta \Phi(x) - D^\beta \Phi_j(x)| \\&\leq \lim\limits_{j \to \infty} \sup_{ x \in K}|D^\beta \Phi(x) - D^\beta \Phi_j(x)|=0
 \end{align*}
 It follows that $D[\R_I^n] \subset S{D^p}[\R_I^n]$ as a continuous embedding for $ 1 \leq p \leq \infty.$ Thus by the Hahn-Banach theorem we see that the Schwartz distributions $D^{'}[\R_I^n] \subset (S{D^p}[\R_I^n)^{'}$ for $ 1 \leq p \leq \infty.$
 \end{proof}



\section{The family $S{D^p}[\R_I^\infty]$:} 
 
    We define the space $SD^p[\mathbb{R}_{I}^{\infty}] $ with the help of the space $SD^p[\mathbb{R}_{I}^{n}],$ using the same approach that led to $L^1[\R_I^\infty].$ We see that $SD^p[\mathbb{R}_{I}^{n}] \subset SD^p[\mathbb{R}_{I}^{n+1}]. $  Thus we can define $SD^p[\widehat{\mathbb{R}}_{I}^{\infty}]= \bigcup\limits_{n=1}^{\infty}SD^p[\mathbb{R}_{I}^{n}].$\\
   
   \begin{Def}
   We say that a measurable function $ f \in S{D^p}[\R_I^\infty]$ if there is a cauchy sequence $\{f_n\} \subset S{D^p}[\widehat{\R}_I^\infty] $ with $f_n \in S{D^p}[\R_I^n]$ and $\lim\limits_{n \to \infty} f_n(x)= f(x), ~\lambda_\infty-$(a.e.)
   \end{Def}
   Theorem 1.5 shows that functions in $S{D^p}[\widehat{\R}_I^\infty] $ differ from functions in its closure $S{D^p}[\R_I^\infty],$ by sets of measure zero.
   \begin{thm}
   $S{D^p}[\widehat{\R}_I^\infty] = S{D^p}[\R_I^\infty].$
   \end{thm}
   \begin{Def}\label{Def31}
     If $ f \in SD^p[\mathbb{R}_{I}^{\infty}],$ we define the integral of $ f$ by $$ \int\limits_{\mathbb{R}_{I}^{\infty}} f(x)d\lambda_{\infty}(x) = \lim\limits_{n \to \infty}\int\limits_{\mathbb{R}_{I}^{n}} f_n(x)d\lambda_{\infty}(x),$$ where $ f_n \in SD^p[\mathbb{R}_{I}^{n}]$ for all $n$  and the family $\{f_n \}$ is a Cauchy sequence.
     \end{Def}
          \begin{thm}
   If $ f \in SD^p[\mathbb{R}_{I}^{\infty}]$, then the integral of $ f$ defined in Definition \ref{Def31} exists  and all theorems that are true for $f \in S{D^p}[\R_I^n]$
also hold for $f \in S{D^p}[\R_I^\infty].$
\end{thm}
\begin{proof} {\bf{ For existance:}} Since the family of functions $\{f_n\}$ is Cauchy, it is follows that if the integral exists, it is unique.  To prove existence, follow the standard argument and first assume that $f(x) \ge 0$.  In this case, the sequence can always be chosen to be increasing, so that the integral exists.  The general case now follows by the standard decomposition.  
\end{proof}
   
          To construct the space $SD^p[\mathbb{R}_{I}^{\infty}],$ for $1\leq p \leq \infty$   \\Choosing a countable dense set of functions $\{\mcE_n(x)\}_{n=1}^{\infty}$ on the unit ball of $L^1[\R_I^\infty] $ and assume $\{\mcE_n^*\}_{n=1}^{\infty} $ be any corresponding set of duality mapping in $L^\infty[\R_I^\infty]$, also if $\mcB $ is $L^1[\R_I^\infty]$ , using Kuelbs lemma, it is clear that the Hilbert space $H$ will contain some non absolute integrable function, we are not sure this non absolute integrable function is HK-integrable or not. Similar argument of Lemma 2.1 in $L^1[\R_I^\infty]$, with assumption $\mcE_k(x) $ by $\mcE_k(x)= (\mcE_k^i(x_1), \mcE_k^i(x_2),..,\mcE_k^i(x_n)$ with $|\mcE_k(x)| < 1, ~x \in \Pi_{j=1}^{n}I_k^i $ and $\mcE_k(x)=0,~x $ not belongs in  $\Pi_{j=1}^{n}I_k^i $. Then $\mcE_k(x) $ is in $L^p[\R_I^\infty]^n= L^p[\R_I^\infty]$ for $ 1 \leq p \leq \infty.$ Define $F_k(.)$ on $L^p[\R_I^\infty]$  and  Let $t_k= \frac{1}{2^k}$ so that $\sum_{k=1}^{\infty}t_k$ is a set of positive numbers that sum to one, define inner product on $L^1[\R_I^\infty]$ by $$<f, g>= \sum_{k=1}^{\infty}t_k\left[\int_{\R_I^\infty}\mathcal{E}_k (x) f(x)d\lambda_\infty(x)\right]\left[\int_{\R_I^\infty}\mathcal{E}_k(x)g(y)d\lambda_\infty(y)\right]^c.$$ 
    Easily we can find that this inner product  and that $$||f||^2=<f, f> = \sum_{k=1}^{\infty}t_k\left|\int_{\R_I^\infty}\mathcal{E}_k(x)f(x)d\lambda_\infty(x)\right|^2.$$
    We call the completion of $L^1[\R_I^\infty]$ with the above inner product is a Hilbert space, which we denote $S{D^2}[\R_I^\infty]$.    
     \begin{thm}\label{thm21}
    For each $p,~1\leq p \leq \infty,$ we have 
    \begin{enumerate}
    \item The space $ L^p[\R_{I}^{\infty}] \subset SD^2[\R_{I}^{\infty}] $ as a continuous, dense and compact embedding.
    \item $\mathfrak{M}[\R_{I}^{\infty}] \subset SD^2[\R_{I}^{\infty}],~\mathfrak{M}[\R_{I}^{\infty}]$ is the space of finitely additive measures on $\R_{I}^{\infty}$, as a continuous dense and compact embedding.
    \end{enumerate}
    \end{thm}
    \begin{proof}
      $(1)$ As $L^p[\R_{I}^{n}] \subset S{D^2}[\R_{I}^{n}],$ for each $p,~1\leq p \leq \infty$ as a continuous, dense and compact embedding. 
      However $ SD^2[\R_{I}^{\infty}]$ is the closure of $ \bigcup_{n=1}^{\infty}SD^2[\R_{I}^{n}].$        It follows $SD^2[\R_{I}^{\infty}] $ contains $\bigcup_{n=1}^{\infty}L^p[\R_{I}^{n}]$ which is dense in $L^p[\R_{I}^{\infty}].$ as it's closure.\\
        $(2) $ As $L^1[\R_{I}^{\infty}] \subset S{D^2}[\R_{I}^{\infty}]$ and $\mathfrak{M}[\R_{I}^{n}] =\{L^1[\R_{I}^{n}]\}^{**}.$\\ It gives  $\bigcup_{n=1}^{\infty}\{\mathfrak{M}[\R_{I}^{n}]\} =\bigcup_{n=1}^{\infty}\{L^1[\R_{I}^{n}]\}^{**}. $ Since $ f \in SD^2[\R_{I}^{\infty}] $ is the limit of a sequence $\{f_n\} \subset \bigcup_{n=1}^{\infty}SD^2[\R_{I}^{n}].$         So $\mathfrak{M}[\R_{I}^{\infty}]= \{L^1[\R_{I}^{\infty}]\}^{**} $ and hence $\mathfrak{M}[\R_{I}^{\infty}] \subset SD^2[\R_{I}^{\infty}].$
    \end{proof}
    \begin{Def} We call $ SD^2[\R_{I}^{\infty}] $ the Jones-strong distribution Hilbert space on $\R_{I}^{\infty}.$ Let $\alpha$ be a multi-index of non negative integers $\alpha=(\alpha_1, \alpha_2,\dots)$ with $|\alpha|= \sum_{j=1}^{\infty}\alpha_j.$ If $\mathcal{D}$ denotes the standard partial differential operator, let $\mathcal{D}^{\alpha}=\mathcal{D}^{\alpha_1}\mathcal{D}^{\alpha}_2\dots .$
    \end{Def}
    \subsubsection{ Test function and Distribution in $R_I^\infty$}
     Here our space is $\R_I^\infty.$ We  replace $\R_I^\infty$ with its support in $\R^n $ of \cite{TG}.  Let $\alpha= (\alpha_1, \alpha_2, \alpha_3,\dots)$ be multi-index of non negative integers, with $|\alpha|=\sum\limits_{k=1}^{\infty}\alpha_k.$\\
     We define the operators $D_{\infty}^{\alpha} $ and $D_{\alpha, \infty}$ by $D_{\infty}^{\alpha}= \Pi_{k=1}^{\infty}\frac{\partial^{\alpha^k}}{\delta x^{\alpha_k}} $ and $D_{\alpha,\infty}= \Pi_{k=1}^{\infty}(\frac{1}{2\pi i}\frac{\partial}{ \partial x_k})^{\alpha_k}.$\\
     Let $C_c[\R_I^\infty] $ be the class of infinitely differentiable functions on $R_I^\infty$ with the compact support and impose the natural locally convex topology $ \tau $ on $C_c[\R_I^\infty] $ to obtain $D[\R_I^\infty].$
     \begin{Def}
      A sequence $\{f_m\} $ converges to $ f \in D[\R_I^\infty] $ with respect to the compact sequential limit topology if and only if there exists a compact set $\mathcal{K} \subset \R_I^\infty,$ which contain the support of $f_m \to f $ for each $m $ and $D_{\infty}^{\alpha}f_m \to D_{\infty}^{\alpha}f $ uniformly on $\mathcal{K},$ for every multi-index $ \alpha \in \mathbb{N}_{0}^{\infty}.$
     \end{Def}
      Let $ u \in C^1[\R_I^\infty] $ and suppose that $\phi \in C_{c}^{\infty}[\R_I^\infty] $ has its support in a unit ball $ B_r, r>0.$\\
      Then $$ \int_{\R_I^\infty}(\phi u y_i)d\lambda_\infty = \int_{\partial B_r}(u \phi)vds - \int_{\R_I^\infty}(u \phi_{y_i})d \lambda_\infty,$$ where $  v $ is the unit outward normal to $B_r$. Since $\phi$ vanishes on the $\partial B_r,$ then $$\int_{\R_I^\infty}(\phi u_{y_i})d\lambda_\infty=-\int_{\R_I^\infty}(u\phi_{y_i})d\lambda_\infty, ~1 \leq i \leq \infty.$$\\
      So, in general case, for any $ u \in C^m[\R_I^\infty] $ and any multi-index $ \alpha = (\alpha_1, \alpha_2,\dots),$  with  $ |\alpha|= \sum_{i=1}^{\infty}\alpha_i= m, $
      \begin{equation}\label{eq1}
      \int_{\R_I^\infty} \phi(D^{\alpha}u)d\lambda_\infty = (-1)^m \int_{\R_I^\infty} u(D^{\alpha}\phi)d\lambda_\infty.
      \end{equation}
      \begin{lem}\label{lem24}
      A function $ u \in L_{loc}^{1}[\R_I^\infty] $ if it is Lebesgue integrable on every compact subset of $\R_I^\infty.$
      \end{lem}
      \begin{proof}
       We know $ u \in L_{loc}^1[\R_I^n] $ if it is Lebesgue integrable on every compact subset of $R_I^n$.\\
       So, $ u \in L_{loc}^{1}[\bigcup\limits_{n=1}^{\infty} \R_I^n] $ if it is Lebesgue integrable on every compact subset of $ \bigcup\limits_{n=1}^{\infty}\R_I^n$.\\
       That is  a function $ u \in L_{loc}^{1}[\R_I^\infty] $ if it is Lebesgue integrable on every compact subset of $\R_I^\infty.$
      \end{proof}
      \begin{rem}
       With the  Lemma(\ref{lem24}), we can conclude the Equation (\ref{eq1}) is fit even if $D^{\alpha}u$ does not exist according to our normal definition.
      \end{rem}
      \begin{Def}
      If $\alpha $ is a multi-index and $ u, v \in L_{loc}^{1}[\R_I^\infty], $ we say that $v$ is the $\alpha^{th}$-weak (or distributional) partial derivative of $u $ and we write $ D^{\alpha}u=v$ provided $$\int_{\R_I^\infty}u(D^{\alpha}\phi)d\lambda_\infty= (-1)^{|\alpha|} \int_{\R_I^\infty} \phi v d\lambda_\infty$$ for all functions $\phi \in C_{c}^{\infty}[\R_I^\infty].$      Thus $ v$ is in the dual space $D^{'}[\R_I^\infty] $ of $D[\R_I^\infty].$
      \end{Def}
      \begin{lem}
       If a weak $\alpha^{th}$-partial derivatives exists for $u,$ then it is unique $\lambda_\infty-$a.e.
      \end{lem}
    \begin{thm}
    $D[\R_{I}^{\infty}] \subset SD^2[\R_{I}^{\infty}]$ as continuous embedding.
    \end{thm}
    \begin{proof}
    Since $D[\R^n] \subset SD^2[\R^n] $ as a continuous embedding. So, $D[\R_{I}^{n}] \subset SD^2[\R_{I}^{n}] $ as a continuous embedding. Clearly by construction of $D[\R_{I}^{\infty}]$ and $SD^2[\R_{I}^{\infty}],$ so easily we can show  $D[\R_{I}^{\infty}] \subset SD^2[\R_{I}^{\infty}] $ as a continuous  embedding.
    \end{proof}
    More analytical way we can state the  above theorem as follows:    
    \begin{thm}\label{thm21}
    Let $D[\R_{I}^{\infty}] $ be $C_{c}^{\infty}[\R_{I}^{\infty}]$ equipped with the standard locally convex topology (test functions). If $\phi_j \to \phi $ in $D[\R_{I}^{\infty}]$, then $\phi_j \to \phi $ in the norm topology of $S{D^2}[\R_{I}^{\infty}],$ so that $D[\R_{I}^{\infty}] \subset S{D^2}[\R_{I}^{\infty}]$ as continuous embedding.
    \end{thm}
    
    \begin{cor}
      Let $D[\R_{I}^{\infty}] $ be $C_{c}^{\infty}[\R_{I}^{\infty}]$ equipped with the standard locally convex topology (test functions). If $\phi_j \to \phi $ in $D[\R_{I}^{\infty}]$, then $\phi_j \to \phi $ in the norm topology of $S{D^2}[\R_{I}^{\infty}],$ so that $D[\R_{I}^{\infty}] \subset S{D^2}[\R_{I}^{\infty}]$ as a dense embedding. 
    \end{cor}
    \begin{proof}
    By the Theorem \ref{thm21}, since $\alpha$ is arbitrary, we see that  $D[\R_{I}^{\infty}] \subset S{D^2}[\R_{I}^{\infty}] $ as a continuous embedding.     Since $C_{c}^{\infty}[\R_{I}^{\infty}] $ is dense in $L^1[\R_{I}^{\infty}]$, so $D[\R_{I}^{\infty}] $ is dense in $S{D^2}[\R_{I}^{\infty}].$
    \end{proof}
    \begin{thm}
    Let $D[\R_{I}^{\infty}] $ be $C_{c}^{\infty}[\R_{I}^{\infty}]$ equipped with the standard locally convex topology (test functions). If  $ T \in D^{'}[\R_{I}^{\infty}] $, then $ T \in  S{D^2}[\R_{I}^{\infty}]^{'}$  so that $ D^{'}[\R_{I}^{\infty}] \subset S{D^2}[\R_{I}^{\infty}]^{'}$ as a continuous dense embedding.
    \end{thm}
    \begin{proof}
    As $D[\R_{I}^{\infty}] $ is locally dense convex subspace of $S{D^2}[\R_{I}^{\infty}],$ then every continuous linear functional, $\mathcal{T}$ defined on $D[\R_{I}^{\infty}], $ can be extended to a continuous linear functional on $S{D^2}[\R_{I}^{\infty}].$\\
    By Riesz representation theorem, every continuous linear functional $\mathcal{T} $ defined on $S{D^2}[\R_{I}^{\infty}] $ is of the form $\mathcal{T}(f) =<f, g >_{SD}, $ for some $ g \in S{D^2}[\R_{I}^{\infty}].$ So, $ \mathcal{T} \in S{D^2}[\R_{I}^{\infty}]^{'}$ and $ \mathcal{T} \leftrightarrow g $ for each $\mathcal{T} \in D^{'}[\R_{I}^{\infty}].$     So, it is possible to map $D^{'}[\R_{I}^{\infty}] $ into $S{D^2}[\R_{I}^{\infty}] $ as a continuous dense embedding.
    \end{proof}
    \begin{thm}
     For any $f,g \in S{D^2}[\R_I^\infty]$ and any multi-index $\alpha,$ we have $$ < D^{\alpha}f, g>_{SD[\R_I^\infty]} = (-i)^{\alpha}< f, g >_{SD[\R_I^\infty]}.$$
    \end{thm}
    \begin{proof}
    Let $\mathcal{E}_k \in C_c^{\infty}[\R_I^\infty]. $ Then for $f \in S{D^2}[\R_I^\infty]$, we have 
    \begin{align*}
    \int_{\R_I^\infty}\mathcal{E}_k(x) D^{\alpha}f(x)d\lambda_\infty(x) &= (-1)^{|\alpha|}\int_{\R_I^\infty} D^{\alpha}\mathcal{E}_k(x)f(x)d\lambda_\infty(x)\\&=(-i)^{\alpha}\int_{\R_I^\infty}\mathcal{E}_k(x)f(x)d\lambda_\infty(x).
    \end{align*}
    Now, for any $ g \in S{D^2}[\R_I^\infty]$, $$< D^{\alpha}f, g>_{SD^{2}[\R_I^\infty]} = (-i)^{\alpha}< f, g >_{S{D^2}}[\R_I^\infty].$$   
    \end{proof}
     \begin{thm}
  The function space $\mathcal{S}[\R_I^\infty] $, of rapid decrease at infinity are contained in $S{D^2}[\R_I^\infty]$ as continuous embedding, so that $\mathcal{S}^{'}[\R_I^\infty] \subset S{D^2}[R_I^\infty]^{'}.$
  \end{thm}
  \begin{proof} 
  Since $\mathcal{S}[\R_I^n] \subset S{D^2}[\R_I^n] $ continuous embedding, so that $\mathcal{S}^{'}[\R_I^n] \subset S{D^2}[\R_I^n]^{'}$.   The  remaining proof is easy, so we left to the reader.
  \end{proof}
     In general we construct the space $S{D^p}[\R_I^\infty]$  for  $f \in L^p[\R_{I}^{\infty}], $ define  $$||f||_{S{D^p}[\R_{I}^{\infty}]} =\left\{\begin{array}{c}\left(\sum\limits_{k=1}^{\infty}t_k\left|\int\limits_{\R_{I}^{\infty}}\mathcal{E}_k(x){D}^{\alpha}f(x)d\lambda_\infty(x)\right|^p\right)^{\frac{1}{p}}, \mbox{~for~} 1\leq p<\infty;\\
  \sup\limits_{k\geq 1}\left|\int\limits_{\R^{\infty}_{I}}\mathcal{E}_k(x){D}^{\alpha}f(x)d\lambda_{\infty}(x)\right|, \mbox{~for~} p=\infty. \end{array}\right.$$  It is easy to see that $||f||_{SD^p[\R_{I}^{\infty}]}$ defines a norm on $ L^p[\R_{I}^{\infty}]$. If $S{D^p}[\R_{I}^{\infty}]$ is completion of $L^p[\R_{I}^{\infty}] $ then we have the following.
  \begin{thm}
   For $S{D^p}[\R_{I}^{\infty}]$, $1 \leq p \leq \infty $
   \begin{enumerate}
  \item If $f_n \to f$ weakly in $L^p[\R_{I}^{\infty}]$ then $f_n \to f $ strongly in $S{D^p}[\R_{I}^{\infty}].$
  \item  $ SD^p[\R_{I}^{\infty}]$ is uniformly convex.
  \item  If $1\leq q\leq\infty$ and  $\frac{1}{p}+ \frac{1}{q}=1 $, then dual space of $ SD^p[\R_{I}^{\infty}] $ is $ SD^q[\R_{I}^{\infty}].$
  \item $S{D^\infty}[\R_I^\infty] \subset S{D^p}[\R_I^\infty].$  
  \end{enumerate}
  \end{thm}
  \begin{proof}
  (1) If $ \{f_n\} $ is weakly convergence in $ L^p[\mathbb{R}_{I}^{\infty}] $ with limit $ f.$ Then $$ \int\limits_{\mathbb{R}_{I}^{n}}\mathcal{E}_k(x)|f_n(x)-f(x)|d\lambda_{\infty}(x) \to 0 \mbox{~for~each~} k.$$   For each $ f_n \in SD^p[\mathbb{R}_{I}^{n}] $ for all $n$, then we have  $$ \lim\limits_{n \to \infty}\int\limits_{\mathbb{R}_{I}^{n}} \mathcal{E}_k(x)|D^\alpha(f_n(x)-f(x))|d\lambda_{\infty}(x) \to 0.$$
(2)  We know $L^p[\mathbb{R}_{I}^{n}] $ is uniformly convex for each $n $ and that is dense and compactly embedded in $SD^q[\mathbb{R}_{I}^{n}] $ for  $1\leq q \leq \infty. $ So, $\bigcup\limits_{n=1}^{\infty}L^p[\mathbb{R}_{I}^{n}]$ is uniformly convex for each $n $ and that is dense and compactly embedded in $\bigcup\limits_{n=1}^{\infty}SD^p[\mathbb{R}_{I}^{n}] $ for  $1\leq p \leq \infty. $ However $L^p[\widehat{\mathbb{R}}_{I}^{\infty}]=\bigcup\limits_{n=1}^{\infty}L^p[\mathbb{R}_{I}^{n}].$ That is  $L^p[\widehat{\mathbb{R}}_{I}^{\infty}] $ is uniformly convex, dense and compactly embedded in $SD^p[\widehat{\mathbb{R}}_{I}^{\infty}] $ for  $1\leq p \leq \infty. $ \\
  As $SD^p[\mathbb{R}_{I}^{\infty}] $ is the closure of $SD^p[\widehat{\mathbb{R}}_{I}^{\infty}].$  Therefore $SD^p[\mathbb{R}_{I}^{\infty}]$ is uniformly convex.\\ 
 (3) From (2) we have  that $ SD^p[\mathbb{R}_{I}^{\infty}] $ is reflexive, for $1<p<\infty.$ Since
 \[
 \left\{ {SD^p[\mathbb{R}_I^k]} \right\}^{\ast} = SD^q[\mathbb{R}_I^k],\;\tfrac{1}{p} + \tfrac{1}{q} = 1,\;\forall k \mbox{~~and~~}
 \]
 \[
 SD^p[\mathbb{R}_I^k] \subset SD^p[\mathbb{R}_I^{k + 1}],\;\forall k\; \Rightarrow \bigcup\limits_{k = 1}^\infty  \left\{ SD^p[\mathbb{R}_I^k] \right\}^{\ast}  = \bigcup\limits_{k = 1}^\infty  SD^q[\mathbb{R}_I^k],\;\tfrac{1}{p} + \tfrac{1}{q} = 1.
 \]
 Since each $f \in SD^p[\mathbb{R}_{I}^{\infty}]$ is the limit of a sequence $\left\{ f_n \right\} \subset \bigcup\limits_{k = 1}^\infty  SD^p[\mathbb{R}_I^k]$, we see that $\left\{ SD^p[\mathbb{R}_I^{\infty}] \right\}^{\ast}=SD^q[\mathbb{R}_I^{\infty}]$, for $\tfrac{1}{p} + \tfrac{1}{q} = 1.$\\
   (4) Let $ f \in SD^{\infty}[\mathbb{R}_{I}^{\infty}].$      This implies $\left|\int\limits_{\mathbb{R}_{I}^{\infty}} \mcE_k(x)\mathcal{D}^{\alpha} f(x) d\lambda_{\infty}(x)\right| $ is uniformly bounded for all $ k.$ It follows that $\left|\int\limits_{\mathbb{R}_{I}^{\infty}} \mcE_k(x)\mcD^{\alpha}f(x) d\lambda_{\infty}(x)\right|^p $ is uniformly bounded for  $ 1\leq p < \infty.$     It is clear from the definition of $SD^p[\mathbb{R}_{I}^{\infty}] $     that 
   \begin{align*}
  \left [\sum\left|\int\limits_{\mathbb{R}_{I}^{\infty}} \mathcal{E}_k(x) D^{\alpha}f(x)d\lambda_{\infty}(x)\right|^p\right]^\frac{1}{p} \leq M||f||_{SD^p[\mathbb{R}_{I}^{\infty}]} < \infty. 
\end{align*}   
So, $ f \in SD^p[\mathbb{R}_{I}^{\infty}].$ 
  \end{proof} 
     
      We recall the space $$X_{p}^{m}[\R^n] =\{B_{\alpha}*g =(1-\Delta)^\frac{-\alpha}{2}g~~: g \in L^p[\R^n], ~~0<\alpha< n , ~0<\alpha< m \}$$ is coincides with $W_p^m[R^n] $ when $1<p< \infty $ and $ m> 0, $ where $B_\alpha $ is the Bessel potential of order $\alpha $, $\Delta$ is the Laplacian and $* $ is the convolution operator.\\
       We define $W_{p}^{m}[\R_I^\infty]$ is the space of all functions $ u \in L_{loc}^{1}[\R_I^\infty] $ whose weak derivative $\partial^{\alpha}u \in L^p[\R_I^\infty]$ for every $\alpha \in \mathbb{N}_{0}^{\infty} $ with $|\alpha|=m.$
    \begin{thm}
    $W_p^m[\R_I^\infty] \subset S{D^2}[\R_I^\infty]$ as a continuous dense embedding, for all $m$ and all $p.$
    \end{thm}
    \begin{proof}
      We can find  $W_p^m[\R_{I}^{n}]\subset SD^2[\R_{I}^{n}] $ as continuous dense embedding.      However $SD^2[\R_{I}^{\infty}] $ is the closure of $\bigcup\limits_{k=1}^{\infty}SD^2[\R_{I}^{k}].$\\
      That is $SD^2[\R_{I}^{\infty}] $  contains  $\bigcup\limits_{k=1}^{\infty}SD^2[\R_{I}^{k}]$ which is dense in $W_p^m[\R_{I}^{\infty}] $ as it's closure.\\
      Hence, $W_p^m[\R_{I}^{\infty}]\subset SD^2[\R_{I}^{\infty}] $ as continuous dense embedding.
      \end{proof}
            In the last, we call a function $f$ such that $\int\limits_{\R_I^\infty}|\mathcal{E}_k(x)f(x)d\lambda_\infty(x)|^p < \infty $ for every compact set $\mathcal{K}$ in $\R_I^\infty$ is said to be in $L_{loc}^{p}[\R_I^\infty].$
           
   \subsubsection{Functions of Bounded variation}
    The objective of this section is to show that every HK-integrable function is in $S{D^2}[\R_I^\infty]$. To do this, we need to discuss a certain class of functions of bounded variation in the sense of Cesari (see \cite{GL}) are well known for working in PDE (partial differential equations) and geometric measure theory. Also we consider the function of bounded variation in Vitali sense (see \cite{TY}) are applied in applied mathematics and engineering for error estimation associated with research in control theory, financial derivatives, robotics, high speed networks and in calculation of certain integrals. 
    We developed this portion through the Definition $3.38 $ and $3.39 $ of \cite{TG}.
    \begin{Def}
     A function $ f \in L^1[\R_I^\infty] $ is said to be bounded variation i.e. $ f\in BV_c[\R_I^\infty]$ if $f \in L^1[\R_I^\infty] $ there exists a signed Radon measure $\mu_i$ such that $$\int\limits_{\R_I^\infty}f(x)\frac{\partial\phi(x)}{\partial x_i}d\lambda_\infty(x)=-\int\limits_{\R_I^\infty}\phi(x)d\mu_i(x_i)$$ for $i=1,2,3,\dots,\infty$ for all $\phi \in C_{0}^{\infty}[\R_I^\infty]$
    \end{Def}
    \begin{Def}
    A function $ f $ with continuous partial derivatives is said to be of bounded variation i.e. $ f \in BV_{v}[\R_I^\infty] $ if for all $D_n=\{(a_i,b_i)\times I_n\}~, 1\leq i \leq n$  for all $(a_i, b_i) $ is an interval in $\R^n,$  $$V(f)=\lim\limits_{n \to \infty}\int_{a_1}^{b_1}\int_{a_2}^{b_2}\dots \int_{a_n}^{b_n}\left|\frac{\partial^nf(x)}{\partial x_1\partial x_2\dots \partial x_n}\right|d \lambda_\infty(x)< \infty.$$
    \end{Def}
    \begin{Def}
     We define $BV_{v, 0}[\R_I^\infty] $ by $$BV_{v, 0}[\R_I^\infty]= \left\{ f(x)\in BV_v[\R_I^\infty]: f(x) \to 0 \mbox{~as~} x_i \to \infty\right\},$$ where $x_i$ is any component of $x.$
    \end{Def}
    \begin{thm}
    The space $HK[\R_I^\infty] $ of all HK-integrable functions is contained in $S{D^2}[\R_I^\infty].$
    \end{thm}
    \begin{proof}
     Since $\mathcal{E}_k(x) $ is continuous and differentiable, therefore $ \mathcal{E}_k(x) \in BV_{v, 0}[\R_I^\infty] $ so that for $ f \in HK[\R_I^\infty] $, gives
     \begin{align*}      
     ||f||_{S{D^2}[\R_I^\infty]} &= \sum\limits_{k=1}^{\infty}t_k\left|\int\limits_{\R_I^\infty}\mathcal{E}_k(x)f(x)d\lambda_\infty(x)\right|^2 \\
     &\leq \sup\limits_{k}\left|\int\limits_{\R_I^\infty}\mathcal{E}_k(x)f(x)d\lambda_\infty(x)\right|^2 \\
     &\leq ||f||_{HK}^{2}[\sup\limits_{k}V(\mathcal{E}_k)]^2<\infty.
     \end{align*}
     So, $ f \in S{D^2}[\R_I^\infty].$
    \end{proof}
    \section{Conclusion}
    We have constructed a new class of separable Banach spaces,
$SD^p [\R_I^\infty],~ 1 \leq p \leq \infty$, which contain each $L^p$-space as a dense continuous and compact embedding. These spaces have the remarkable property
that, for any multi-index $\alpha,~||D^{\alpha}u||_{SD} = ||u||_{SD}.$ We have shown that our
spaces contain the non-absolutely integrable functions and the space of test
functions $D[\R_I^\infty  ],$ as a dense continuous embedding. We have discussed their
basic properties and their relationship to $D^{ '} [\R_I^\infty], ~\mcS[\R_I^\infty]$ and $\mcS^{ ' }[\R_I^\infty].$
\section{Acknowledgement: } The authors would like to thank Prof. Tepper L. Gill for suggesting the formation of $S{D^p}[\R_I^\infty]$ and
making valuable suggestions  that improve the presentation of the paper.
\bibliographystyle{amsalpha}

\end{document}